\documentclass{article}
\title{The sum of the unitary divisor function}
\author{Tim Trudgian\footnote{Supported by Australian Research Council DECRA Grant DE120100173.}\\ Mathematical Sciences Institute\\ The Australian National University, ACT 0200, Australia\\ timothy.trudgian@anu.edu.au }

\usepackage{url}
\usepackage{amsthm}
\usepackage{amsmath}
\usepackage{amssymb}
\usepackage{booktabs}
\newtheorem{thm}{Theorem}
\newtheorem{cor}{Corollary}

\begin{document}

\maketitle

\begin{abstract}
\noindent This article establishes a new upper bound on the function $\sigma^{*}(n)$, the sum of all coprime divisors of $n$. The main result is that $\sigma^{*}(n)\leq 1.3007 n \log \log n$ for all $n\geq 570,571$.
\end{abstract}

\section{Introduction}
\subsection{The function $\sigma(n)$}\label{S1.1}
Let $\sigma(n)$ denote the sum of the divisors of $n$; for example, $\sigma(12) = 1+ 2+ 3+ 4+ 6+ 12 = 28.$ 
In 1913 Gr\"{o}nwall showed that 
\begin{equation}\label{p1}
\limsup \sigma(n)/(n \log\log n) = e^{\gamma} = 1.78107\ldots,
\end{equation}
where $\gamma$ is Euler's constant. A proof is given in \cite[Thm.\ 322]{HW}. Robin showed that the manner in which (\ref{p1}) behaves is connected with the Riemann hypothesis. More precisely, he showed, in \cite{Robineg}, that for $n\geq 5041$ the inequality
\begin{equation}\label{rc}
\sigma(n) < e^{\gamma} n \log\log n
\end{equation}
is equivalent to the Riemann hypothesis.
Ivi\'{c} \cite{IvicSum} showed that
\begin{equation*}\label{ivic1}
\sigma(n) < 2.59 n \log\log n, \quad (n\geq 7),
\end{equation*}
which was improved by Robin [\textit{op.\ cit.}] to
\begin{equation}\label{robin1}
\sigma(n) < \frac{\sigma(12)}{12\log\log 12} n \log\log n \leq 2.5634n\log\log n, \quad (n\geq 7).
\end{equation}
Akbary, Friggsted and Juricevic \cite{AFJ} improved this further, replacing the right-side of (\ref{robin1}) with 
\begin{equation}\label{aft}
\frac{\sigma(180)}{180 \log\log 180}n \log\log n \leq 1.8414n \log\log n \leq 1.0339e^{\gamma}n \log\log n, \quad (n \geq 121).
\end{equation}
Given Robin's criterion for the Riemann hypothesis in (\ref{rc}) it is reasonable to suggest that (\ref{aft}) is close to the best bound that one may hope to exhibit.
\subsection{The function $\sigma^{*}(n)$}\label{S1.2}
We say that $d$ is a unitary divisor of $n$ if $d|n$ and $(d, n/d) =1$. Let $\sigma^{*}(n) = \sum_{d|n, (d, n/d) =1} d$ be the sum of all unitary divisors of $n$; for example, $\sigma^{*}(12) = 1+ 12 + 3 + 4 = 20$.
Robin \cite[p.\ 210]{Robineg} notes that the proof of (\ref{p1}) can be adapted to show that
\begin{equation}\label{p2}
\limsup \sigma^{*}(n)/(n \log\log n) = \frac{6e^{\gamma}}{\pi^{2}}= 1.08\ldots,
\end{equation}
see also \cite[p.\ 21]{IvicSum}.
Ivi\'{c} \cite{IvicSum} showed that
\begin{equation*}\label{ivic2}
\sigma^{*}(n) < \frac{28}{15} n  \log\log n, \quad (n\geq 31).
\end{equation*}
This was improved by Robin who showed that
\begin{equation*}\label{rob2}
\sigma^{*}(n) < 1.63601 n  \log\log n, \quad (n\geq 31),
\end{equation*}
except for $n=42$ when $\sigma^{*}(n) = 1.7366\ldots n\log\log n$. A direct comparison of these results with those in \S \ref{S1.1} compels us to ask the following questions.
\begin{enumerate}
\item Given (\ref{p2}) can a Robin-esque criterion for the Riemann hypothesis \`{a} la (\ref{rc}) be given for $\sigma^{*}(n)$? \label{a}
\item Analogous to (\ref{aft}) can one obtain a relatively close approximation to (\ref{p2}) of the form
\begin{equation*}
\sigma^{*}(n) < (1+ \epsilon) \frac{6e^{\gamma}}{\pi^{2}} n  \log\log n, \quad (n\geq n_{0}),
\end{equation*}
for reasonably small values of $\epsilon$ and $n_{0}$? \label{b}
\end{enumerate}

Concerning \ref{a}, Robin has conjectured \cite[Prop.\ 1 (i), p.\ 210]{Robineg} that there are infinitely many $n$ for which $$\sigma^{*}(n) > \frac{6 e^{\gamma}}{\pi^{2}} n \log\log n.$$ A related conjecture is given in Proposition 1 (ii) in \cite{Robineg}, viz.\ that 
\begin{equation}\label{f1}
\frac{\sigma(n)}{\sigma^{*}(n) \log\log n} < e^{\gamma},
\end{equation} 
for all $n$ sufficiently large. The interest in this conjecture stems from the limiting relation $$\lim\sup\frac{\sigma(n)}{\sigma^{*}(n) \log\log n} = e^{\gamma}.$$ Derbal \cite{Derbal} proved (\ref{f1}) for all $n \geq 17$.

This article answers Question \ref{b} above, at least partially, by proving
\begin{thm}\label{maine}
For $n\geq 570,571$,
\begin{equation}\label{main}
\sigma^{*}(n) \leq 1.3007 n \log\log n.
\end{equation}
\end{thm}
It takes less than 40 seconds on a 1.8GHz laptop to compute $\sigma^{*}(n)$ for all $1\leq n \leq 570,570$. One may therefore justify the number 570,571 appearing in Theorem \ref{maine} as being `reasonably small', as stipulated in Question 2, as least in regards to computational resources. 

It would be of interest to address the following problem. Fix an $\epsilon>0$ and determine the least value of $n_{0}$ such that $\sigma^{*}(n) < (1+ \epsilon) \frac{6e^{\gamma}}{\pi^{2}} n  \log\log n$ for all $n\geq n_{0}$. The method used to prove Theorem \ref{maine} is incapable of reducing the right-side of (\ref{main}) to anything less than $1.29887 n \log\log n$.

Theorem \ref{maine} is proved in \S \ref{pet}. An application is given in \S \ref{bbb}. Two concluding questions are raised in \S \ref{four}.
\section{Proof of Theorem \ref{maine}}\label{pet}
We proceed as in Robin \cite[p.\ 211]{Robineg}. It is sufficient to verify the inequality on numbers $N_{k} = \prod_{i=1}^{k} p_{i}$, where $k\geq 2$, since, for $N_{k} \leq n < N_{k+1},$ we have $\sigma^{*}(n)/n \leq \sigma^{*}(N_{k})/N_{k}$, whence
\begin{equation}\label{gg}
\frac{\sigma^{*}(n)}{n \log \log n} \leq \frac{\sigma^{*}(N_{k})}{N_{k} \log \log N_{k}}.
\end{equation}
Since $\sigma^{*}(p^{\alpha}) = 1+ p^{\alpha}$ and $\sigma^{*}(n)$ is a multiplicative function, the right-side of (\ref{gg}) is
\begin{equation}\label{gh}
\frac{\prod_{i\leq k} \left( 1+p_{i}^{-1}\right)}{\log \theta(p_{k})},
\end{equation}
where $\theta(x) = \sum_{p \leq x} \log p$.
To bound the numerator in (\ref{gh}) we use
\begin{equation*}\label{dus1}
\sum_{p\leq x} \frac{1}{p} \leq \log\log x + B + \frac{1}{10 \log^{2}x} + \frac{4}{15\log^{3} x}, \quad(x\geq 10,372),
\end{equation*}
where $$ B = \gamma + \sum_{p\geq 2} \left\{\log\left(1 - \frac{1}{p}\right) + \frac{1}{p}\right\} = 0.26149\ldots,$$
see Dusart \cite{Dusart1999}. To bound the denominator in (\ref{gh}) we use
\begin{equation*}\label{dus2}
\theta(x) \geq x\left( 1 - \frac{0.006788}{\log x}\right), \quad (x \geq 10,544,111),
\end{equation*}
which is also found in \cite{Dusart1999}.
Therefore, since $e^{x} \geq x+1$ we have
\begin{equation*}
\prod_{i\leq k} \left( 1 + \frac{1}{p_{i}}\right) \leq \exp\left( \sum_{i\leq k} \frac{1}{p_{i}}\right)\leq A_{1}(p_{k}) \log p_{k},
\end{equation*}
where 
\begin{equation*}
A_{1}(x) = \exp\left(B + \frac{1}{10 \log^{2} x} + \frac{4}{15 \log^{2} x}\right), \quad (x \geq 10,372).
\end{equation*}
Also
\begin{equation*}
\log\theta(p_{k}) \geq A_{2}(p_{k})\log p_{k} ,
\end{equation*}
where 
\begin{equation*}
A_{2}(x) = 1 + \frac{\log(1 - 0.006788/\log x)}{\log x}, \quad (x\geq 10,544,111).
\end{equation*}
 It is clear that 
\begin{equation}\label{lop}
A_{2}(x) < 1< e^{B} = 1.29887\ldots < A_{1}(x).
\end{equation}
We choose a suitably large lower bound on $k$ in order to make $A_{1}(x)$ and $A_{2}(x)$ sufficiently close to $e^{B}$ and $1$ respectively.
Indeed, we shall bound (\ref{gh}) for $p_{k} \geq 15,485,863$, which is equivalent to $k\geq 1,000,000$. Therefore 
\begin{equation}\label{19}
\frac{\prod_{i\leq k} \left( 1+p_{i}^{-1}\right)}{\log \theta(p_{k})}\leq \frac{A_{1}(p_{k})}{A_{2}(p_{k})} \leq 1.3007,
\end{equation}
whence 
\begin{equation}\label{main2}
\frac{\sigma^{*}(n)}{n \log \log n} \leq \frac{\sigma^{*}(N_{k})}{N_{k} \log \log N_{k}} \leq 1.3007,
\end{equation}
for all $k\geq 10^6.$
One may check that (\ref{main2}) also holds for $8\leq k \leq 10^{6}.$ On a single core PC with 32 GB of RAM, this calculation took less than a minute using \textit{Magma}. All that remains are the numbers $3\leq n \leq p_{1}\cdots p_{8} = 9,699,690.$ A quick computational check shows that
$$ \frac{\sigma^{*}(570,570)}{570,570 \log \log 570,570} \geq 1.3125,$$
and that, for all $n> 570,570$, the inequality (\ref{main}) holds, which proves Theorem \ref{maine}. Were this lower bound on $n$ too large for one's tastes, one could also show
\begin{equation*}\label{main3}
\sigma^{*}(n) \leq 1.3007 n \log\log n,
\end{equation*} 
for all $n\geq 53,131$ with only two exceptions, namely
\begin{equation*}
\begin{split}
\sigma^{*}(510,510) &= (1.3245\ldots) 510,510 \log\log 510,510, \quad \textrm{and} \\
\sigma^{*}(570,570) &= (1.3125\ldots) 570,570 \log\log 570,570.
\end{split}
\end{equation*}
Our bounds for $\sigma^{*}(n)$ depend on an upper bound for $A_{1}(p_{k})/A_{2}(p_{k})$ in (\ref{19}). We see at once from (\ref{lop}) that our method is incapable of reducing the bound 1.3007 in Theorem \ref{maine} to anything below 1.29887.
\section{Application to exponential divisors}\label{bbb}
%Minculete \cite{Minc} has applied (\ref{ivic2}) to exponential  divisors. 
Given an $n = p_{1}^{a_{1}}\cdots p_{s}^{a_{s}}$ the integer $d=p_{1}^{b_{1}}\cdots p_{s}^{b_{s}}$ is an \textit{exponential divisor} of $n$ if $b_{j}|a_{j}$ for every $1 \leq j\leq s$. %For example, the exponential divisors of $16 = 2^{4}$ are $2, 4 (= 2^{2})$, and  $16 (= 2^{4})$. 
Define the functions $d^{(e)}(n)$ and $\sigma^{(e)}(n)$ to be the number of exponential divisors of $n$ and the sum of the exponential divisors of $n$, respectively. Since these functions are multiplicative we have
\begin{equation*}\label{bb}
d^{(e)}(n) = \prod_{j=1}^{r}d(a_{j}), \quad \sigma^{(e)}(n) = \prod_{j=1}^{r} \left( \sum_{b_{j}|a_{j}} p_{j}^{b_{j}}\right),
\end{equation*}
where $d(n)$ is the number of divisors of $n$. Minculete \cite[Thm. 2.1 and Cor. 2.5]{Minc} has given the following bounds for $\sigma^{(e)}(n)$ and $d(n) d^{(e)}(n)$ 
\begin{equation*}\label{minx1}
\sigma^{(e)}(n) \leq \frac{28}{15} n \log\log n, \quad (n \geq 6),
\end{equation*}
\begin{equation*}\label{minx2}
d^{(e)}(n)d(n) \leq \frac{28}{15}n \log\log n, \quad (n \geq 5).
\end{equation*}
An application of the proof of Theorem \ref{maine} improves these bounds.
\begin{cor}
For $n \geq 37$,
\begin{equation}\label{cor1}
\sigma^{(e)}(n) \leq 1.3007 n \log\log n.
\end{equation}
For $n \geq 8$,
\begin{equation}\label{cor2}
d^{(e)}(n)d(n) \leq 1.3007 n \log\log n.
\end{equation}
\end{cor}
\begin{proof}
The displayed formula halfway down page 1529 in \cite{Minc} gives
$$ \sigma^{(e)} \leq n \prod_{p|n} \left( 1+ \frac{1}{p}\right),$$
so that
\begin{equation}\label{cook}
\frac{\sigma^{(e)}(n)}{n \log \log n} \leq \frac{\prod_{p|n} \left( 1+ \frac{1}{p}\right)}{\log\log n}.
\end{equation}
As before, we need only consider (\ref{cook}) on $N_{k}\leq n < N_{k+1}$. Using (\ref{19}) and the calculations in \S \ref{pet} we have
$$\frac{\sigma^{(e)}(n)}{n \log \log n} \leq 1.3007, \quad(n \geq 9,699,691).$$
Checking the range $37\leq n \leq 9,699,691$ establishes (\ref{cor1}).
Minculete \cite[Eq.~(12)]{Minc} showed that $d(n) d^{(e)}(n) \leq \sigma^{(e)}(n)$ for all $n\geq 1$. Using this, (\ref{cor1}), and a simple computer check for $8\leq n \leq 36$, establishes~(\ref{cor2}).
\end{proof}

\section{Conclusion}\label{four}
Both of the functions $\sigma^{*}(n)$ and $\sigma^{(e)}(n)$ are multiplicative. We have
$$ \sigma^{*}(p) = 1+ p > \sigma^{(e)}(p) = p,$$
and, for $a\geq 2$,
$$\sigma^{*}(p^{a}) = 1+ p^{a} < p+ p^{a} \leq \sigma^{(e)}(p^{a}),$$
since $a = a \cdot 1$, where $a$ and $1$ are distinct. Therefore, on square-free numbers $\sigma^{*}(n) > \sigma^{(e)}(n)$. We conclude this section by raising two questions.
\begin{enumerate}
\item What is the proportion of $n$ for which $\sigma^{*}(n)> \sigma^{(e)}(n)$?
\item Are there infinitely many values of $n$ for which $\sigma^{*}(n) = \sigma^{(e)}(n)$?
\end{enumerate}
The proportion in Question 1 must be at least that of the square-free numbers, viz.\ $6/\pi^{2} \approx 0.607.$ A computation shows the proportion of $1\leq n \leq 10^{9}$ to be approximately 0.778307. It follows from the Erd\H{o}s--Wintner theorem (see, e.g., \cite[III.4]{Tenbook}) that the density of $n$ for which $\sigma^{*}(n)> \sigma^{(e)}(n)$ is well defined. In \cite{Deleglise} the density of the set of integers $n$ for which $\sigma(n)/n \geq 2$ was estimated. It seems possible that similar methods may be brought to bear on Question 1.

As for Question 2, only five values of $n$ were found in the range $1\leq n \leq 10^{9}$ for which $\sigma^{*}(n) = \sigma^{(e)}(n)$, namely
$$ n = 20, 45, 320, 6615, 382200.$$
Andrew Lelechenko has also found 
$$n= 680890228200,$$
which is the next smallest $n$ after 382200. He has also communicated to me that $\sigma^{*}(n) = \sigma^{(e)}(n)$ also for 
$$n=2456687209744634987008753664 = 2^{49} \times
4363953127297.$$

\subsection*{Acknowledgements}
I am grateful to Danesh Jogia who verified (\ref{main2}) for $8\leq k \leq 10^{6}$, Scott Morrison who provided a much-needed tutorial on programming,  Greg Martin for a discussion on limiting distributions, and Andrew Lelechenko for providing the last examples in \S 4.

\bibliographystyle{plain}
\bibliography{themastercanada}

\end{document}